\documentclass[11pt,a4paper]{amsart}
\usepackage{amsmath,amssymb,color,multicol,setspace}
\usepackage[utf8,utf8x]{inputenc}
\usepackage{hyperref}
\usepackage{longtable}
\usepackage{xy,amscd}
\usepackage{epsfig}
\usepackage{rotating}
\usepackage{xspace}
\xyoption{all}

\newtheorem{propo}{Proposition}[section]
\newtheorem{corol}[propo]{Corollary}

\newtheorem{lemma}[propo]{Lemma}
\newtheorem{conje}[propo]{Conjecture}
\newtheorem{algor}[propo]{Algorithm}

\theoremstyle{definition}
\newtheorem{defin}[propo]{Definition}
\newtheorem{examp}[propo]{Example}

\theoremstyle{remark}
\newtheorem{remar}[propo]{Remark}

\numberwithin{equation}{section}

\newcommand{\algo}[6]
{\vspace{11pt}\addtocounter{propo}{1}
\begin{algor}{{\bf #1}{\rm (#2)}}\label{#1}\end{algor}
\vspace{-6pt}\noindent{\it #3}.\\
{\bf Input:} #4\\
{\bf Output:} #5\\
\newcounter{#1}
\begin{list}{\textbf{\arabic{#1}.}}{\usecounter{#1}}
#6\end{list}\vspace{3pt}}

\newcommand{\NN }{\mathbb{N}}

\newcommand{\RR }{\mathbb{R}}

\newcommand{\ZZ }{\mathbb{Z}}

\DeclareMathOperator{\SL}{SL}

\newcommand{\Fc }{\mathcal F}

\newcommand{\QB }{Q}

\newcommand\generic{specialized\xspace}
\newcommand\Generic{Specialized\xspace}
\newcommand\kfrieze{$\SL_k$-frieze\xspace}
\newcommand\kfriezes{$\SL_k$-friezes\xspace}

\title[On wild frieze patterns]
{On wild frieze patterns}

\author{M.~Cuntz}
\address{Michael Cuntz,
Institut f\"ur Algebra, Zahlentheorie und Diskrete Mathematik,
Fakult\"at f\"ur Mathematik und Physik,
Leibniz Universit\"at Hannover,
Welfengarten 1,
D-30167 Hannover, Germany}
\email{cuntz@math.uni-hannover.de}

\begin{document}

\begin{abstract}
In this note, among other things, we show:
\begin{enumerate}
\item There are periodic wild $\SL_k$-frieze patterns whose entries are positive integers.
\item There are non-periodic $\SL_k$-frieze patterns whose entries are positive integers.
\item There is an $\SL_3$-frieze pattern whose entries are positive integers and with infinitely many different entries.
\end{enumerate}
\end{abstract}

\maketitle

Conway-Coxeter frieze patterns are arrays of positive integers bounded by diagonals of ones such that every adjacent $2\times 2$ subdeterminant is one. They were introduced in \cite{hsmC71} and classified for the first time in \cite{jChC73}. More than $30$ years later, they appeared in a natural way in several other contexts, for instance in the fields of cluster algebras and of Nichols algebras.

Among the generalizations of frieze patterns that have been introduced since their first appearance (see \cite{MG15}) are the $\SL_k$-frieze patterns which we will discuss in more detail in this note.
In the literature, $\SL_k$-frieze patterns satisfy an $\SL_k$-condition (adjacent $k\times k$ subdeterminants are $1$), usually they consist of integers, and sometimes the entries are required to be positive. For example, the array
%
\[
\begin{array}{cccccccccccc}
1 & 0 & 0 & 1 & 13 & 88 & 314 & 25 & 4 & 1 & 0 & 0\\
2 & 1 & 0 & 0 & 1 & 7 & 25 & 2 & 1 & 2 & 1 & 0\\
138 & 72 & 1 & 0 & 0 & 1 & 4 & 1 & 49 & 138 & 72 & 1\\
389 & 203 & 3 & 1 & 0 & 0 & 1 & 2 & 138 & 389 & 203 & 3\\
203 & 106 & 2 & 3 & 1 & 0 & 0 & 1 & 72 & 203 & 106 & 2\\
3 & 2 & 3 & 17 & 7 & 1 & 0 & 0 & 1 & 3 & 2 & 3\\
1 & 3 & 17 & 97 & 40 & 6 & 1 & 0 & 0 & 1 & 3 & 17\\
0 & 1 & 7 & 40 & 17 & 6 & 13 & 1 & 0 & 0 & 1 & 7\\
0 & 0 & 1 & 6 & 6 & 25 & 88 & 7 & 1 & 0 & 0 & 1
\end{array}
\]
is an $\SL_3$-frieze pattern. In the case of tame friezes, i.e.\ when all adjacent $(k+1)\times(k+1)$-subdeterminants are zero, these three conditions appear to be very natural; on the one hand they imply that the entries of the frieze are specializations of variables of a cluster algebra, on the other hand, positivity has a nice geometric interpretation. Integrality could turn out to be natural because of some yet unknown representation theoretical motivation (as for example given by the Nichols algebras in the case of Conway-Coxeter friezes, see \cite{p-CH09d}, \cite{p-C14}).

However, in the wild (not tame) case, the examples presented in this note suggest that the usual conditions on the entries do not seem to be the natural ones. In fact, the determinants of the adjacent $(k-1)\times(k-1)$ submatrices in an $\SL_k$-frieze pattern will probably play a more important role than the entries of the frieze themselves.
And we believe that to assume that these determinants are nonzero (the frieze is then tame) is more natural than to assume positivity of the entries.

We exhibit some facts on tame friezes in the second section.
In the last section we provide many examples of wild $\SL_k$-friezes illustrating that it is probably hopeless to classify all wild friezes (even when the entries are positive integers), although
to my knowledge, no wild friezes with positive integral entries were known before.

\medskip
\noindent{\bf Acknowledgement:}
{I would like to thank C.\ Bessenrodt, T.\ Holm, P.\ J\o rgensen, and S.\ Morier-Genoud for calling my attention to some of the questions addressed here and for many further valuable comments.}

\section{$\SL_k$-frieze patterns}

$\SL_k$-frieze patterns have been considered for the first time in \cite{CR72} with a slightly more special definition.
Meanwhile, many papers have treated different types of $\SL_k$-patterns.
Let us shortly recapitulate some definitions and results.

\begin{defin} Let $k\in\NN$, $k\ge 2$.
An \emph{$\SL_k$-frieze pattern} (or \emph{\kfrieze}for short) is an array $\Fc$ of the form
\[
\begin{array}{ccccccccccccc}
 & & & & \ddots & & &  & & & & & \\
0 & \cdots & 0 & 1 & c_{i-1,1} & \cdots & c_{i-1,n} & 1 & 0 & \cdots & 0 & & \\
& 0 & \cdots & 0 & 1 & c_{i,1} & \cdots & c_{i,n} & 1 & 0 & \cdots & 0 & \\
& & 0 & \cdots & 0 & 1 & c_{i+1,1} & \cdots & c_{i+1,n} & 1 & 0 & \cdots & 0 \\
 & & & & & & & & \ddots & & & & 
\end{array}
\]
where there are $k-1$ zeros on the left and right in each row, $c_{i,j}$ are numbers in some field, and such that every
(complete) adjacent $k\times k$ submatrix has determinant $1$.
It may be convenient to extend the pattern to the right and to the left by repeating the rows; in this case one has to multiply every second repetition by $\varepsilon:=(-1)^{k-1}$:
\[
\cdots \:\:\: 0 \:\:\: 1 \:\:\: c_{i-1,1} \:\:\: \cdots \:\:\: c_{i-1,n} \:\:\: 1 \:\:\: \underbrace{0 \:\:\: \cdots \:\:\: 0}_{k-1} \:\:\:
\varepsilon \:\:\: \varepsilon c_{i-1,1} \:\:\: \cdots \:\:\: \varepsilon c_{i-1,n} \:\:\: \varepsilon \:\:\: 0 \:\:\: \cdots
\]

We call $n$ the \emph{height} of $\Fc$.
We call the frieze $\Fc$
\begin{itemize}
\item \emph{integral} if all $c_{i,j}$ are in $\ZZ$,
\item \emph{non-zero} if $c_{i,j}\ne 0$ for all $i,j$,
\item \emph{positive} if all $c_{i,j}$ are positive real numbers,
\item \emph{periodic} if there exists an $m>0$ such that $c_{i,j}=c_{i+m,j}$ for all $i$ and $1\le j\le n$,
\item \emph{\generic}if any adjacent $(k-1)\times(k-1)$ submatrix not containing too many zeros from the border has nonzero determinant,
\item \emph{tame} if all adjacent $(k+1)\times(k+1)$ submatrices have determinant $0$,
\item \emph{wild} if it is not tame.
\end{itemize}
Sometimes it will be useful to view the extended $\Fc$ as a matrix $\Fc = (a_{i,j})_{i,j\in\ZZ}$, where $a_{1,1}=c_{1,1}$ and thus $a_{i,j+i-1}=c_{i,j}$.
\end{defin}

\begin{examp}
\begin{enumerate}
\item Conway-Coxeter friezes are exactly the integral positive $\SL_2$-friezes. They are all tame, \generic, and periodic.
\item
Extending the array
\[
\begin{array}{ccccccccccc}
1 & 0 & 0 & 1 & 3 & 8 & 4 & 7 & 1 & 0 & 0\\
1 & 1 & 0 & 0 & 1 & 3 & 2 & 4 & 1 & 1 & 0\\
4 & 7 & 1 & 0 & 0 & 1 & 3 & 8 & 4 & 7 & 1\\
2 & 4 & 1 & 1 & 0 & 0 & 1 & 3 & 2 & 4 & 1\\
3 & 8 & 4 & 7 & 1 & 0 & 0 & 1 & 3 & 8 & 4\\
1 & 3 & 2 & 4 & 1 & 1 & 0 & 0 & 1 & 3 & 2\\
0 & 1 & 3 & 8 & 4 & 7 & 1 & 0 & 0 & 1 & 3\\
0 & 0 & 1 & 3 & 2 & 4 & 1 & 1 & 0 & 0 & 1
\end{array}
\]
periodically (in all directions) we obtain a tame integral positive $\SL_3$-frieze.
\end{enumerate}
\end{examp}

\section{Tame frieze patterns}

\subsection{$\xi$-sequences}
It has been proved several times in the literature that tame friezes are periodic, see for example \cite{MGOST14}.

The main reason is (compare \cite{BR10}): Let $\Fc=(a_{i,j})$ be a tame \kfrieze
and denote
\[ F_{i,j}^\ell := \begin{pmatrix}
a_{i,j} & \cdots & a_{i,j+\ell-1} \\
\vdots &  & \vdots \\
a_{i+\ell-1,j} & \cdots & a_{i+\ell-1,j+\ell-1}
\end{pmatrix}. \]
Then the matrices
$B_j:=F_{i,j}^{-1} F_{i,j+1}$
are all equal for fixed $j$ and $i\in\ZZ$, and they are of the form
\[
\xi(c_1,\ldots,c_{k-1}) = \begin{pmatrix}
0 & \cdots & \cdots &\cdots & 0 & (-1)^{k+1} \\
1 & 0 & \cdots &\cdots & 0 & (-1)^{k} c_1 \\
0 & 1 & 0 & \cdots & 0 & \vdots \\
\vdots & & \ddots & \ddots & \vdots & \vdots \\
\vdots & & & \ddots & 0 & -c_{k-2} \\
0 & \cdots & & \cdots & 1 & c_{k-1}
\end{pmatrix},
\]
where the last column has alternating signs and $c_1,\ldots,c_{k-1}\in\RR_{\ge 0}$ if the frieze is positive.

Since $\Fc$ is tame, it is periodic with period $n+k+1$. Thus the sequence of $B_j$ is periodic, and
$B_1\cdots B_{n+k+1}=(-1)^{k-1} I$. Further, the complete frieze pattern is uniquely determined by this sequence of $B_j$, thus by a sequence of $n+k+1$ tuples of the form $(c_1,\ldots,c_{k-1})$.

An interesting question concerning tame friezes is:

\begin{conje}
There are only finitely many tame integral positive $\SL_3$-friezes of height $4$.
More precisely, the number of such friezes could be $26952$.
\end{conje}

Notice that there are only finitely many tame integral positive $\SL_2$-friezes (see for example \cite{jChC73}), that there are only finitely many tame integral positive $\SL_3$-friezes of height less than $4$ ($5$, $51$, $868$ of heights $2,3,4$ resp.), but that there are infinitely many tame integral positive $\SL_3$-friezes of height greater than $4$ (see \cite{MG12}).
Notice also that the examples in Section \ref{wfp} show that there are infinitely many wild integral positive $\SL_k$-friezes of a fixed height.

\subsection{\Generic friezes}
The (apparently) most useful identity in the context of \kfriezes is Sylvester's identity:
If we denote $D_{i,j}^\ell$ the determinant of the $\ell\times\ell$ adjacent submatrix at position $(i,j)$ in the array $(a_{i,j})_{i,j\in\ZZ}$ (i.e.\ $D_{i,j}^\ell=\det F_{i,j}^\ell$), then
\begin{equation}\label{impeq}
D_{i,j}^{\ell+1} D_{i+1,j+1}^{\ell-1} = D_{i,j}^\ell D_{i+1,j+1}^\ell-D_{i+1,j}^\ell D_{i,j+1}^\ell,
\end{equation}
for $\ell\ge 1$ (assume $D_{i,j}^0=1$).
\begin{corol}\label{cornerpos}
If $D_{i,j}^\ell,D_{i+1,j}^\ell,D_{i,j+1}^\ell,D_{i,j}^{\ell+1},D_{i+1,j+1}^{\ell-1}$ are positive, then $D_{i+1,j+1}^\ell$ is positive as well.
\end{corol}
\begin{corol}
\Generic \kfriezes are tame by
Equation (\ref{impeq}) for $\ell=k$:
\begin{equation*}
D_{i,j}^{k+1} D_{i+1,j+1}^{k-1} = D_{i,j}^k D_{i+1,j+1}^k-D_{i+1,j}^k D_{i,j+1}^k = 1-1 = 0.
\end{equation*}
\end{corol}

If $\Fc$ is a \generic \kfrieze, then it is uniquely determined by any sequence of $k-1$ successive rows.
\begin{examp}
If we choose $c_{i,j}=a_{i,j+i-1}$, $1\le i < k$, $1\le j \le n$ to be transcendent and algebraically independent, then $\Fc$ is \generic.
\end{examp}
\begin{defin}
Replacing every entry $a_{i,j}$ in $\Fc$ by $D_{i,j}^{k-1}$ we obtain the \emph{dual} frieze $\hat\Fc$.
\end{defin}
If $\Fc$ is tame, then $\hat\Fc$ is a transposed, translated copy of $\Fc$ (see \cite{CR72} or \cite{MGOST14}).
Thus:
\begin{corol}
A tame frieze is \generic if and only if it is nonzero.
\end{corol}

But is there a tame frieze which is not \generic? Yes:

\begin{examp}
Extending the array
\[
\begin{array}{cccccccccc}
1 & 0 & 0 & 1 & 0 & 0 & 1 & 1 & 0 & 0\\
1 & 1 & 0 & 0 & 1 & 1 & 1 & 1 & 1 & 0\\
1 & 1 & 1 & 0 & 0 & 1 & 2 & 1 & 1 & 1\\
1 & 0 & 0 & 1 & 0 & 0 & 1 & 1 & 0 & 0\\
1 & 0 & -1 & 1 & 1 & 0 & 0 & 1 & 0 & -1\\
0 & 1 & 0 & -1 & 1 & 1 & 0 & 0 & 1 & 0\\
0 & 0 & 1 & 0 & -1 & 0 & 1 & 0 & 0 & 1
\end{array}
\]
periodically (in all directions) we obtain a tame but not \generic integral $\SL_3$-frieze.
\end{examp}

The following lemma (together with the fact that tame friezes are periodic) generalizes the main result of \cite{CR72}.

\begin{lemma}
If $\Fc$ is an \kfrieze such that
\begin{equation}\label{lemass}
D_{1,j}^\ell>0 \quad \text{ for all } \ell=1,\ldots,k-1 \text{ and } j=1,\ldots,n,
\end{equation}
then $\Fc$ is \generic (and thus tame and periodic).
\end{lemma}
\begin{proof}
It suffices to prove (see (i) and (ii) below) that (\ref{lemass}) holds for the next sequence of $k-1$ rows of the pattern:
Induction then gives $D_{i,i+j-1}^\ell>0$ for $\ell=1,\ldots,k-1$, $j=1,\ldots,n$, and $i\ge 1$.
But then (\ref{lemass}) also holds for the frieze given by $(a_{k+1-i,k+1-j})_{i,j}$. The same argument thus shows that
$D_{i,i+j-1}^\ell>0$ for $\ell=1,\ldots,k-1$, $j=1,\ldots,n$, and all $i\in\ZZ$.
In particular, the case $\ell=k-1$ implies that $\Fc$ is \generic.

(i) Let $\ell=2$. Then $D_{1,1}^{\ell-1}$, $D_{1,2}^{\ell-1}$, and $D_{1,1}^{\ell}$ are positive by (\ref{lemass}) (where $D_{1,1}^{\ell}=1$ if $\ell=k$), $D_{2,2}^{\ell-2}=1$, and $D_{2,1}^{\ell-1}=1>0$ because it is on the border of the frieze.
Thus by Corollary \ref{cornerpos} for $i=1$, $j=1$, $\ell=2$, we get $D_{2,2}^{1}>0$.
Induction over $\ell$ yields $D_{2,2}^{\ell-1}>0$ for $3\le\ell\le k$ as well (notice that we always have $D_{2,1}^{\ell-1}=1>0$).

(ii) But then we can move to the next entries to the right, $D_{2,j}^{\ell-1}$, $j\ge 3$: As before,
$D_{1,j-1}^{\ell-1}$, $D_{1,j}^{\ell-1}$, and $D_{1,j-1}^{\ell}$ are positive by (\ref{lemass}),
and $D_{2,j-1}^{\ell-1}>0$ by the induction in (i).
Hence by an induction over $j$ and for each $j$ an induction over $\ell$ (like in (i)) we obtain $D_{2,j}^{\ell-1}>0$ for all $j=2,\ldots,n+1$ and $2\le\ell\le k$.
\end{proof}

\section{Wild frieze patterns}\label{wfp}

Are there wild \kfriezes? Yes:

\begin{examp}
Extending the array
\[
\begin{array}{cccccccc}
1 & 1 & 1 & 2 & 1 & 1 & 0 & 0\\
0 & 1 & 1 & 1 & 2 & 4 & 1 & 0\\
0 & 0 & 1 & 1 & 1 & 2 & 1 & 1\\
1 & 0 & 0 & 1 & 1 & 1 & 2 & 4\\
1 & 1 & 0 & 0 & 1 & 1 & 1 & 2\\
2 & 4 & 1 & 0 & 0 & 1 & 1 & 1
\end{array}
\]
periodically (in all directions) we obtain a wild periodic integral positive $\SL_3$-frieze.
\end{examp}

\subsection{Aperiodic friezes}
But what about non-periodic wild friezes?
The key difference between tame and wild friezes is: If $\Fc$ is a wild \kfrieze, then there exist $k-1$ successive rows which do not uniquely determine the following row by the $\SL_k$-condition. This happens if and only if some adjacent $(k-1)\times(k-1)$ submatrix within these rows has determinant zero.

Define a directed graph $\Gamma_n$ which has the set of $(k-1)$-tuples of vectors in $\NN^n$ as vertices and such that two tuples
$\bar v = (v_1,\ldots,v_{k-1})$ and $\bar w = (w_1,\ldots,w_{k-1})$ are connected by an edge $\bar v \rightarrow \bar w$ if
$(v_2,\ldots,v_{k-1})=(w_1,\ldots,w_{k-2})$ and if $(v_1,\ldots,v_{k-1},w_{k-1})$ satisfy the $\SL_k$-condition when embedded into a frieze.

The following example provides non-periodic friezes with finitely many different entries.

\begin{examp}
In this example, $k=3$. We choose $12$ vertices in the graph $\Gamma_5$:
Let
$$A_{1}:=
\begin{array}{cccccccc}
1 & 1 & 1 & 1 & 1 & 1 & 1 & 0\\
0 & 1 & 2 & 2 & 1 & 1 & 2 & 1
\end{array}
,\quad A_{2}:=
\begin{array}{cccccccc}
1 & 1 & 1 & 1 & 1 & 1 & 1 & 0\\
0 & 1 & 2 & 2 & 1 & 2 & 4 & 1
\end{array}
,$$ $$ A_{3}:=
\begin{array}{cccccccc}
1 & 1 & 1 & 1 & 1 & 2 & 1 & 0\\
0 & 1 & 2 & 2 & 1 & 2 & 2 & 1
\end{array}
,\quad A_{4}:=
\begin{array}{cccccccc}
1 & 1 & 2 & 1 & 1 & 2 & 1 & 0\\
0 & 1 & 4 & 2 & 1 & 2 & 2 & 1
\end{array}
,$$ $$ A_{5}:=
\begin{array}{cccccccc}
1 & 2 & 1 & 1 & 1 & 1 & 1 & 0\\
0 & 1 & 1 & 2 & 1 & 1 & 2 & 1
\end{array}
,\quad A_{6}:=
\begin{array}{cccccccc}
1 & 2 & 1 & 1 & 2 & 1 & 1 & 0\\
0 & 1 & 1 & 1 & 1 & 1 & 2 & 1
\end{array}
,$$ $$A_{7}:=
\begin{array}{cccccccc}
1 & 2 & 1 & 1 & 2 & 2 & 1 & 0\\
0 & 1 & 1 & 1 & 1 & 1 & 1 & 1
\end{array}
,\quad A_{8}:=
\begin{array}{cccccccc}
1 & 2 & 2 & 1 & 1 & 2 & 1 & 0\\
0 & 1 & 2 & 2 & 1 & 2 & 2 & 1
\end{array}
,$$ $$ A_{9}:=
\begin{array}{cccccccc}
1 & 2 & 2 & 1 & 2 & 2 & 1 & 0\\
0 & 1 & 2 & 1 & 1 & 1 & 1 & 1
\end{array}
,\quad A_{10}:=
\begin{array}{cccccccc}
1 & 2 & 2 & 1 & 2 & 2 & 1 & 0\\
0 & 1 & 2 & 1 & 1 & 2 & 2 & 1
\end{array}
,$$ $$ A_{11}:=
\begin{array}{cccccccc}
1 & 2 & 2 & 1 & 2 & 4 & 1 & 0\\
0 & 1 & 2 & 1 & 1 & 2 & 1 & 1
\end{array}
,\quad A_{12}:=
\begin{array}{cccccccc}
1 & 4 & 2 & 1 & 2 & 2 & 1 & 0\\
0 & 1 & 1 & 1 & 1 & 1 & 1 & 1
\end{array}
.$$
Now putting the pieces $A_1,\ldots,A_{12}$ together, where successive pieces overlap with one row and according to the following subgraph of $\Gamma_5$,
\medskip
\begin{center}
\includegraphics[scale=0.5]{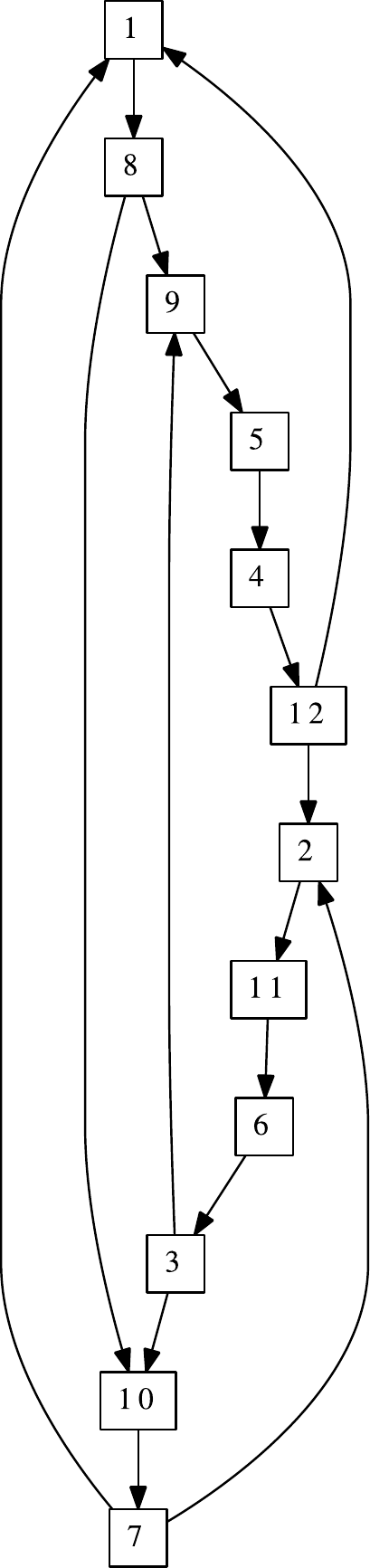}
\end{center}
\medskip
will always yield patterns satisfying the $\SL_3$ condition. In particular, choosing different loops in an aperiodic way will produce an integral positive $\SL_3$-frieze which is not periodic.
\end{examp}

\begin{examp}
Wild friezes really are wild. Figure \ref{fig_wildgraphs} consists of two further examples of subgraphs of $\Gamma_n$. Every closed path corresponds to a periodic integral positive $\SL_3$-frieze. As above, we may construct infinitely many non periodic friezes from these pictures. (Of course, we have omitted a list of the corresponding vertices.)
\begin{figure}
\begin{center}
\begin{tabular}{cc}
\includegraphics[scale=1.2]{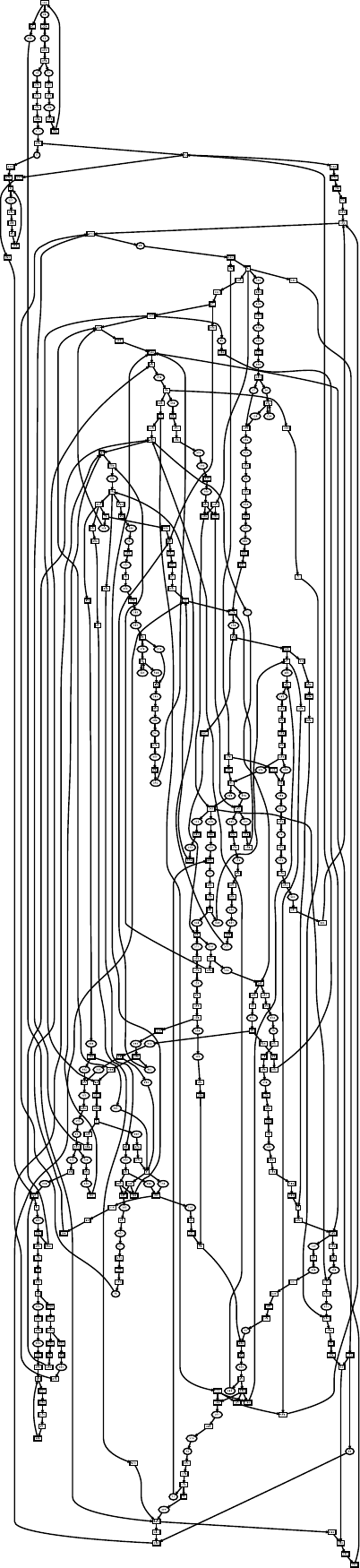}
&
\includegraphics[scale=1.2]{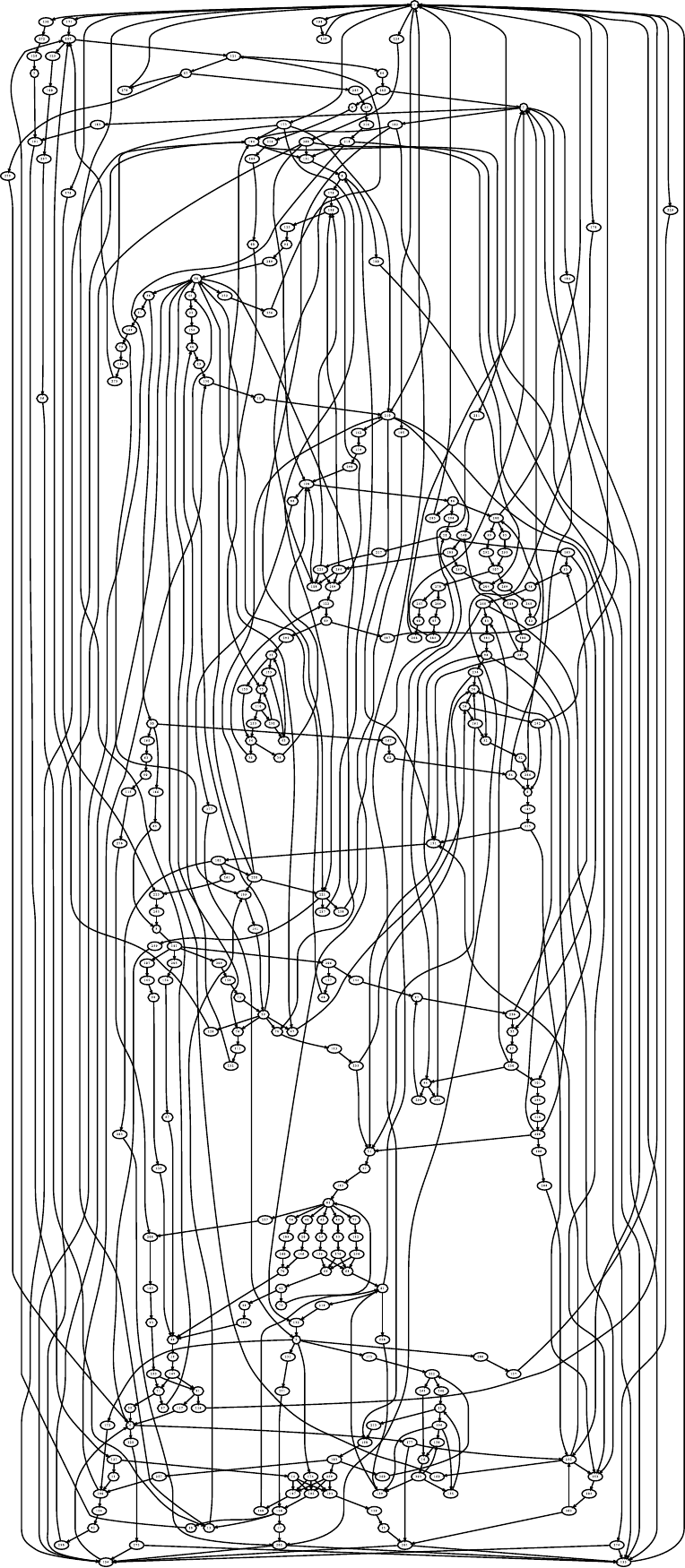}
\end{tabular}
\end{center}
\caption{Examples of graphs producing wild friezes.}\label{fig_wildgraphs}
\end{figure}
\end{examp}
\begin{figure}[htbp]
  \begin{sideways}
    \begin{minipage}{20cm}
\begin{tiny}
\[ \QB_\ell =
\begin{array}{ccccccccccccccccccccc}
1 & 4 & a_{1} & a_{2} & a_{3} & a_{4} & a_{5} & a_{6} & a_{7} & 1 & 0 & 0 & 1 & 4 & a_{1} & a_{2} & a_{3} & a_{4} & a_{5} & a_{6} & a_{7}\\
0 & 1 & a_{8} & a_{9} & a_{10} & a_{11} & a_{12} & a_{13} & a_{14} & 1 & 1 & 0 & 0 & 1 & a_{8} & a_{9} & a_{10} & a_{11} & a_{12} & a_{13} & a_{14}\\
0 & 0 & 1 & a_{15} & a_{16} & 4 & a_{17} & a_{18} & 4 & a_{19} & a_{20} & 1 & 0 & 0 & 1 & a_{15} & a_{16} & 4 & a_{17} & a_{18} & 4\\
1 & 0 & 0 & 1 & 5 & a_{21} & a_{22} & a_{23} & a_{24} & a_{25} & a_{26} & a_{27} & 1 & 0 & 0 & 1 & 5 & a_{21} & a_{22} & a_{23} & a_{24}\\
1 & 1 & 0 & 0 & 1 & a_{28} & a_{29} & a_{30} & a_{31} & a_{32} & a_{33} & a_{18} & 1 & 1 & 0 & 0 & 1 & a_{28} & a_{29} & a_{30} & a_{31}\\
a_{34} & a_{35} & 1 & 0 & 0 & 1 & a_{36} & a_{37} & 4 & a_{38} & a_{39} & 4 & a_{34} & a_{35} & 1 & 0 & 0 & 1 & a_{36} & a_{37} & 4\\
a_{40} & a_{41} & a_{42} & 1 & 0 & 0 & 1 & 2 & a_{43} & a_{44} & a_{45} & a_{46} & a_{40} & a_{41} & a_{42} & 1 & 0 & 0 & 1 & 2 & a_{43}\\
a_{47} & a_{48} & a_{49} & 1 & 1 & 0 & 0 & 1 & a_{50} & a_{51} & a_{52} & a_{53} & a_{47} & a_{48} & a_{49} & 1 & 1 & 0 & 0 & 1 & a_{50}\\
a_{54} & a_{55} & 4 & a_{56} & a_{57} & 1 & 0 & 0 & 1 & a_{58} & a_{1} & 4 & a_{54} & a_{55} & 4 & a_{56} & a_{57} & 1 & 0 & 0 & 1\\
a_{59} & a_{60} & a_{61} & a_{62} & a_{63} & a_{64} & 1 & 0 & 0 & 1 & 2 & a_{65} & a_{59} & a_{60} & a_{61} & a_{62} & a_{63} & a_{64} & 1 & 0 & 0\\
a_{66} & a_{67} & a_{68} & a_{69} & a_{70} & a_{19} & 8 & 1 & 0 & 0 & 1 & a_{71} & a_{66} & a_{67} & a_{68} & a_{69} & a_{70} & a_{19} & 8 & 1 & 0\\
a_{72} & a_{21} & 4 & a_{73} & a_{74} & 4 & a_{75} & a_{76} & 1 & 0 & 0 & 1 & a_{72} & a_{21} & 4 & a_{73} & a_{74} & 4 & a_{75} & a_{76} & 1
\end{array},
\]
\bigskip
\bigskip
\bigskip
\[ \QB_0 =
\begin{array}{ccccccccccccccccccccc}
1 & 4 & 9 & 60 & 160 & 29 & 45 & 18 & 20 & 1 & 0 & 0 & 1 & 4 & 9 & 60 & 160 & 29 & 45 & 18 & 20\\
0 & 1 & 4 & 27 & 72 & 13 & 20 & 8 & 9 & 1 & 1 & 0 & 0 & 1 & 4 & 27 & 72 & 13 & 20 & 8 & 9\\
0 & 0 & 1 & 7 & 19 & 4 & 8 & 3 & 4 & 4 & 7 & 1 & 0 & 0 & 1 & 7 & 19 & 4 & 8 & 3 & 4\\
1 & 0 & 0 & 1 & 5 & 5 & 21 & 7 & 12 & 25 & 45 & 7 & 1 & 0 & 0 & 1 & 5 & 5 & 21 & 7 & 12\\
1 & 1 & 0 & 0 & 1 & 2 & 9 & 3 & 5 & 10 & 18 & 3 & 1 & 1 & 0 & 0 & 1 & 2 & 9 & 3 & 5\\
5 & 8 & 1 & 0 & 0 & 1 & 5 & 2 & 4 & 9 & 16 & 4 & 5 & 8 & 1 & 0 & 0 & 1 & 5 & 2 & 4\\
21 & 35 & 5 & 1 & 0 & 0 & 1 & 2 & 7 & 20 & 35 & 12 & 21 & 35 & 5 & 1 & 0 & 0 & 1 & 2 & 7\\
12 & 20 & 3 & 1 & 1 & 0 & 0 & 1 & 4 & 12 & 21 & 7 & 12 & 20 & 3 & 1 & 1 & 0 & 0 & 1 & 4\\
8 & 13 & 4 & 7 & 16 & 1 & 0 & 0 & 1 & 5 & 9 & 4 & 8 & 13 & 4 & 7 & 16 & 1 & 0 & 0 & 1\\
5 & 8 & 3 & 6 & 14 & 1 & 1 & 0 & 0 & 1 & 2 & 2 & 5 & 8 & 3 & 6 & 14 & 1 & 1 & 0 & 0\\
20 & 32 & 11 & 21 & 49 & 4 & 8 & 1 & 0 & 0 & 1 & 7 & 20 & 32 & 11 & 21 & 49 & 4 & 8 & 1 & 0\\
3 & 5 & 4 & 10 & 23 & 4 & 23 & 5 & 1 & 0 & 0 & 1 & 3 & 5 & 4 & 10 & 23 & 4 & 23 & 5 & 1
\end{array},
\]
\bigskip
\bigskip
\[
\QB_1 =
\begin{array}{ccccccccccccccccccccc}
1 & 4 & 29 & 84 & 192 & 41 & 261 & 58 & 12 & 1 & 0 & 0 & 1 & 4 & 29 & 84 & 192 & 41 & 261 & 58 & 12\\
0 & 1 & 12 & 35 & 80 & 17 & 108 & 24 & 5 & 1 & 1 & 0 & 0 & 1 & 12 & 35 & 80 & 17 & 108 & 24 & 5\\
0 & 0 & 1 & 3 & 7 & 4 & 32 & 7 & 4 & 44 & 75 & 1 & 0 & 0 & 1 & 3 & 7 & 4 & 32 & 7 & 4\\
1 & 0 & 0 & 1 & 5 & 49 & 437 & 95 & 68 & 833 & 1421 & 19 & 1 & 0 & 0 & 1 & 5 & 49 & 437 & 95 & 68\\
1 & 1 & 0 & 0 & 1 & 18 & 161 & 35 & 25 & 306 & 522 & 7 & 1 & 1 & 0 & 0 & 1 & 18 & 161 & 35 & 25\\
81 & 128 & 1 & 0 & 0 & 1 & 9 & 2 & 4 & 61 & 104 & 4 & 81 & 128 & 1 & 0 & 0 & 1 & 9 & 2 & 4\\
3317 & 5243 & 41 & 1 & 0 & 0 & 1 & 2 & 107 & 1804 & 3075 & 148 & 3317 & 5243 & 41 & 1 & 0 & 0 & 1 & 2 & 107\\
1860 & 2940 & 23 & 1 & 1 & 0 & 0 & 1 & 60 & 1012 & 1725 & 83 & 1860 & 2940 & 23 & 1 & 1 & 0 & 0 & 1 & 60\\
112 & 177 & 4 & 123 & 280 & 1 & 0 & 0 & 1 & 17 & 29 & 4 & 112 & 177 & 4 & 123 & 280 & 1 & 0 & 0 & 1\\
1053 & 1664 & 47 & 1598 & 3638 & 13 & 1 & 0 & 0 & 1 & 2 & 34 & 1053 & 1664 & 47 & 1598 & 3638 & 13 & 1 & 0 & 0\\
3564 & 5632 & 159 & 5405 & 12305 & 44 & 8 & 1 & 0 & 0 & 1 & 115 & 3564 & 5632 & 159 & 5405 & 12305 & 44 & 8 & 1 & 0\\
31 & 49 & 4 & 170 & 387 & 4 & 411 & 89 & 1 & 0 & 0 & 1 & 31 & 49 & 4 & 170 & 387 & 4 & 411 & 89 & 1
\end{array}
.\]
\end{tiny}
    \end{minipage}
  \end{sideways}
  \centering
  \caption{The segment $\QB_\ell$ of the pattern and as examples $\QB_0$ and $\QB_1$.}
  \label{Q0Q1}
\end{figure}

\subsection{Unbounded entries}
We now display a $\SL_3$-frieze which is integral, positive, not periodic, and with unbounded entries.
Let $w=9+\sqrt{80}$, and denote by $\QB_\ell$ the array displayed in Figure \ref{Q0Q1}
where the entries $a_1(\ell),\ldots,a_{76}(\ell)$ are given in Figure \ref{entri}.
\begin{figure}
\begin{tiny}
\begin{multicols}{2}
\begin{spacing}{0.5}
\begin{align*}
a_{1}&=1/40(13w + 63)w^{-\ell}+1/40(-13w + 297)w^\ell\\
a_{2}&=1/40(113w + 3)w^{-2 \ell}+9+1/40(-113w + 2037)w^{2 \ell}\\
a_{3}&=1/10(77w + 1)w^{-2 \ell}+106/5+1/10(-77w + 1387)w^{2 \ell}\\
a_{4}&=1/8(11w + 17)w^{-\ell}+1/8(-11w + 215)w^\ell\\
a_{5}&=1/40(73w + 27)w^{-2 \ell}+54/5+1/40(-73w + 1341)w^{2 \ell}
\end{align*}
\begin{align*}
a_{6}&=1/40(31w + 5)w^{-2 \ell}+19/5+1/40(-31w + 563)w^{2 \ell}\\
a_{7}&=1/20(21w + 11)w^{-\ell}+1/20(-21w + 389)w^\ell\\
a_{8}&=1/20(3w + 13)w^{-\ell}+1/20(-3w + 67)w^\ell\\
a_{9}&=1/40(51w + 1)w^{-2 \ell}+4+1/40(-51w + 919)w^{2 \ell}\\
a_{10}&=1/40(139w + 1)w^{-2 \ell}+47/5+1/40(-139w + 2503)w^{2 \ell}
\end{align*}
\begin{align*}
a_{11}&=1/8(5w + 7)w^{-\ell}+1/8(-5w + 97)w^\ell\\
a_{12}&=1/40(33w + 11)w^{-2 \ell}+23/5+1/40(-33w + 605)w^{2 \ell}\\
a_{13}&=1/20(7w + 1)w^{-2 \ell}+8/5+1/20(-7w + 127)w^{2 \ell}\\
a_{14}&=1/40(19w + 9)w^{-\ell}+1/40(-19w + 351)w^\ell\\
a_{15}&=1/8(3w + 1)w^{-\ell}+1/8(-3w + 55)w^\ell
\end{align*}
\begin{align*}
a_{16}&=1/40(41w + 11)w^{-\ell}+1/40(-41w + 749)w^\ell\\
a_{17}&=1/4(w + 7)w^{-\ell}+1/4(-w + 25)w^\ell\\
a_{18}&=1/8(w + 3)w^{-\ell}+1/8(-w + 21)w^\ell\\
a_{19}&=1/20(-w + 49)w^{-\ell}+1/20(w + 31)w^\ell\\
a_{20}&=1/40(-3w + 167)w^{-\ell}+1/40(3w + 113)w^\ell
\end{align*}
\begin{align*}
a_{21}&=1/40(-w + 109)w^{-\ell}+1/40(w + 91)w^\ell\\
a_{22}&=1/10(5w + 13)w^{-2 \ell}+47/5+1/10(-5w + 103)w^{2 \ell}\\
a_{23}&=1/40(9w + 11)w^{-2 \ell}+12/5+1/40(-9w + 173)w^{2 \ell}\\
a_{24}&=1/4(w + 15)w^{-\ell}+1/4(-w + 33)w^\ell\\
a_{25}&=1/40(15w + 101)w^{-2 \ell}+66/5+1/40(-15w + 371)w^{2 \ell}
\end{align*}
\begin{align*}
a_{26}&=1/10(7w + 43)w^{-2 \ell}+119/5+1/10(-7w + 169)w^{2 \ell}\\
a_{27}&=1/40(11w + 41)w^{-\ell}+1/40(-11w + 239)w^\ell\\
a_{28}&=w^{-\ell}+w^\ell\\
a_{29}&=1/40(9w + 19)w^{-2 \ell}+4+1/40(-9w + 181)w^{2 \ell}\\
a_{30}&=1/10(w + 1)w^{-2 \ell}++1/10(-w + 19)w^{2 \ell}
\end{align*}
\begin{align*}
a_{31}&=1/8(w + 11)w^{-\ell}+1/8(-w + 29)w^\ell\\
a_{32}&=1/40(7w + 37)w^{-2 \ell}+5+1/40(-7w + 163)w^{2 \ell}\\
a_{33}&=1/40(13w + 63)w^{-2 \ell}+9+1/40(-13w + 297)w^{2 \ell}\\
a_{34}&=1/40(-9w + 181)w^{-\ell}+1/40(9w + 19)w^\ell\\
a_{35}&=1/20(-7w + 143)w^{-\ell}+1/20(7w + 17)w^\ell
\end{align*}
\begin{align*}
a_{36}&=1/40(9w + 19)w^{-\ell}+1/40(-9w + 181)w^\ell\\
a_{37}&=1/10(w + 1)w^{-\ell}+1/10(-w + 19)w^\ell\\
a_{38}&=1/8(w + 27)w^{-\ell}+1/8(-w + 45)w^\ell\\
a_{39}&=1/4(w + 23)w^{-\ell}+1/4(-w + 41)w^\ell\\
a_{40}&=1/10(-5w + 103)w^{-2 \ell}+47/5+1/10(5w + 13)w^{2 \ell}
\end{align*}
\begin{align*}
a_{41}&=1/40(-31w + 651)w^{-2 \ell}+82/5+1/40(31w + 93)w^{2 \ell}\\
a_{42}&=1/40(w + 91)w^{-\ell}+1/40(-w + 109)w^\ell\\
a_{43}&=1/40(-11w + 239)w^{-\ell}+1/40(11w + 41)w^\ell\\
a_{44}&=1/40(-3w + 223)w^{-2 \ell}+51/5+1/40(3w + 169)w^{2 \ell}\\
a_{45}&=1/10(-w + 95)w^{-2 \ell}+89/5+1/10(w + 77)w^{2 \ell}
\end{align*}
\begin{align*}
a_{46}&=1/4(-w + 33)w^{-\ell}+1/4(w + 15)w^\ell\\
a_{47}&=1/40(-11w + 231)w^{-2 \ell}+27/5+1/40(11w + 33)w^{2 \ell}\\
a_{48}&=1/40(-17w + 365)w^{-2 \ell}+47/5+1/40(17w + 59)w^{2 \ell}\\
a_{49}&=1/40(w + 51)w^{-\ell}+1/40(-w + 69)w^\ell\\
a_{50}&=1/20(-3w + 67)w^{-\ell}+1/20(3w + 13)w^\ell
\end{align*}
\begin{align*}
a_{51}&=1/40(-w + 125)w^{-2 \ell}+31/5+1/40(w + 107)w^{2 \ell}\\
a_{52}&=1/40(-w + 213)w^{-2 \ell}+54/5+1/40(w + 195)w^{2 \ell}\\
a_{53}&=1/8(-w + 37)w^{-\ell}+1/8(w + 19)w^\ell\\
a_{54}&=1/4(-w + 25)w^{-\ell}+1/4(w + 7)w^\ell\\
a_{55}&=1/8(-3w + 79)w^{-\ell}+1/8(3w + 25)w^\ell
\end{align*}
\begin{align*}
a_{56}&=1/8(-3w + 55)w^{-\ell}+1/8(3w + 1)w^\ell\\
a_{57}&=1/20(-17w + 313)w^{-\ell}+1/20(17w + 7)w^\ell\\
a_{58}&=1/40(7w + 37)w^{-\ell}+1/40(-7w + 163)w^\ell\\
a_{59}&=1/40(-7w + 131)w^{-2 \ell}+8/5+1/40(7w + 5)w^{2 \ell}\\
a_{60}&=1/40(-11w + 207)w^{-2 \ell}+13/5+1/40(11w + 9)w^{2 \ell}
\end{align*}
\begin{align*}
a_{61}&=1/8(-w + 21)w^{-\ell}+1/8(w + 3)w^\ell\\
a_{62}&=1/40(-11w + 199)w^{-2 \ell}++1/40(11w + 1)w^{2 \ell}\\
a_{63}&=1/40(-25w + 453)w^{-2 \ell}+13/5+1/40(25w + 3)w^{2 \ell}\\
a_{64}&=1/40(-w + 29)w^{-\ell}+1/40(w + 11)w^\ell\\
a_{65}&=1/10(-w + 19)w^{-\ell}+1/10(w + 1)w^\ell
\end{align*}
\begin{align*}
a_{66}&=1/40(-23w + 443)w^{-2 \ell}+41/5+1/40(23w + 29)w^{2 \ell}\\
a_{67}&=1/10(-9w + 175)w^{-2 \ell}+66/5+1/10(9w + 13)w^{2 \ell}\\
a_{68}&=1/8(-3w + 71)w^{-\ell}+1/8(3w + 17)w^\ell\\
a_{69}&=1/40(-37w + 673)w^{-2 \ell}+4+1/40(37w + 7)w^{2 \ell}\\
a_{70}&=1/10(-21w + 383)w^{-2 \ell}+51/5+1/10(21w + 5)w^{2 \ell}
\end{align*}
\begin{align*}
a_{71}&=1/40(-13w + 257)w^{-\ell}+1/40(13w + 23)w^\ell\\
a_{72}&=1/40(-w + 69)w^{-\ell}+1/40(w + 51)w^\ell\\
a_{73}&=1/2(-w + 19)w^{-\ell}+1/2(w + 1)w^\ell\\
a_{74}&=1/8(-9w + 173)w^{-\ell}+1/8(9w + 11)w^\ell\\
a_{75}&=1/40(-51w + 919)w^{-\ell}+1/40(51w + 1)w^\ell\\
a_{76}&=1/40(-11w + 199)w^{-\ell}+1/40(11w + 1)w^\ell
\end{align*}
\end{spacing}
\end{multicols}
\end{tiny}
  \caption{The entries $a_1(\ell),\ldots,a_{76}(\ell)$.}
  \label{entri}
\end{figure}
See Figure \ref{Q0Q1} for examples.
One can check that concatenating $$\ldots,\QB_{-2},\QB_{-1},\QB_0,\QB_1,\QB_2,\ldots$$ produces an integral, positive $\SL_3$-frieze in which the entries increase without limit.
The integrality comes from the fact that the above formulas are expressions for integral recursions. For example,
\[ t_\ell : =a_{28} = w^{-\ell}+w^\ell, \]
satisfies
\[ t_\ell = 18 t_{\ell-1} - t_{\ell-2}, \quad t_1=18, \quad t_0=2. \]

\begin{remar}
We computed this example with the following technique.
Start with an arbitrary vertex in $\Gamma_n$ and successively construct adjacent vertices (with entries within a certain bound).
This will produce a graph with loops and many branches. Some of them turn out to be leaves or ``dead ends'' (at least we find no continuation within the bounds).
Now, after removing some dead ends, we look for a longest path in our constructed subgraph. Then we look for a pattern among the vertices on the path.
\end{remar}


\def\cprime{$'$}
\providecommand{\bysame}{\leavevmode\hbox to3em{\hrulefill}\thinspace}
\providecommand{\MR}{\relax\ifhmode\unskip\space\fi MR }
\providecommand{\MRhref}[2]{%
  \href{http://www.ams.org/mathscinet-getitem?mr=#1}{#2}
}
\providecommand{\href}[2]{#2}

\end{document}